\documentclass[seceqn,number,citesort,dvips]{arxbj}

\aid{0}
\volume{17}
\issue{1}
\pubyear{2011}
\firstpage{226}
\lastpage{252}
\doi{10.3150/10-BEJ271}

\makeatletter

\newproclaim{Definition}{Definition}[section]

\newtheorem{Theorem}{Theorem}[section]
\newtheorem{Lemma}[Theorem]{Lemma}
\newtheorem{Proposition}[Theorem]{Proposition}

\newremark{Remark}{Remark}[section]
\newremark{Example}{Example}

\def\cnvg{\stackrel{v}{\to}}

\def\boldf{\mathbf{f}}
\def\bX{\mathbf{X}}
\def\bY{\mathbf{Y}}

\def\bZ{\mathbf{Z}}

\def\bzero{\mathbf{0}}

\def\bx{\mathbf{x}}

\def\binfty{\bolds{\infty}}

\def\E{\mathbb{E}}
\def\bM{\mathbb{M}}
\def\P{\mathbf{P}}

\def\R{\mathbb{R}}
\def\C{\mathfrak{C}}
\def\D{\mathfrak{D}}

\def\finv{f^{\leftarrow}}
\def\tmu{\tilde{\mu}}
\newcommand{\eqref}[1]{(\ref{#1})}
\def\Enl{\mathbb{E}_{\sqsupset}}
\def\Enb{\mathbb{E}_{\sqcap}}
\def\Snb{S_{\sqcap}}
\def\nunb{\mathbb{\nu}_{\sqcap}}
\def\Ecal{\mathcal{E}}

\makeatother

\begin{document}
\begin{frontmatter}

\title{Conditioning on an extreme component:
Model consistency with regular variation on cones}
\runtitle{Conditioning on an extreme component}

\begin{aug}
\author[a]{\fnms{Bikramjit} \snm{Das}\corref{}\thanksref{a}\ead[label=e1]{bikram@math.ethz.ch}} \and
\author[b]{\fnms{Sidney I.} \snm{Resnick}\thanksref{b}\ead[label=e2]{sir1@cornell.edu}}
\runauthor{B. Das and S.I. Resnick}
\address[a]{RiskLab, Department of Mathematics, ETH Z\"{u}rich, 8092 Z\"{u}rich,
Switzerland.\\ \printead{e1}}
\address[b]{School of Operations Research and Information Engineering,
Cornell University, Ithaca, NY 14853, USA. \printead{e2}}
\end{aug}

\received{\smonth{5} \syear{2008}}
\revised{\smonth{1} \syear{2010}}

%
\begin{abstract}
Multivariate extreme value theory assumes a
multivariate domain of attraction condition for the distribution
of a random vector. This necessitates that each component satisfies a
marginal domain of attraction condition.
An approximation of the joint distribution of a random
vector obtained by conditioning on one of the components being extreme
was developed by Heffernan and Tawn \cite{heffernantawn2004}
and further studied by Heffernan and Resnick \cite{heffernanresnick2007}.
These papers left unresolved the consistency of different models
obtained by conditioning on different components being extreme and
we here provide clarification of this issue.
We also clarify the relationship between these conditional distributions,
multivariate extreme value theory and standard regular variation
on cones of the form $[0,\infty]\times(0,\infty]$.
\end{abstract}

%

\end{frontmatter}

\section{Introduction}\label{sec:intro}

Classical multivariate extreme value theory (abbreviated as MEVT)
captures the extremal dependence structure between components
under a robust multivariate domain of attraction condition which requires
that each marginal distribution belongs to the (maximum) domain of
attraction (hereafter abbreviated as DOA) of some univariate extreme
value distribution.
Extremal dependence has been well studied,
both in the case of asymptotic dependence
and
asymptotic independence
\cite
{dehaanresnick1977,pickands1981,dehaanferreira2006,resnickbook2008,ledfordtawn1996,ledfordtawn1997,maulikresnick2005,resnickbook2007,ramosledford2009,maulikresnick2005,resnick2002a,resnick2008}.
An innovative approach was provided by Heffernan and Tawn \cite
{heffernantawn2004}, who
approximated multivariate distributions by assuming that only
one of the components was in an extreme value domain of
attraction and that this component was extreme. Their approach allowed
for a variety
of examples of different types of asymptotic dependence and
asymptotic independence.
Their statistical
ideas were given a more mathematical framework by Heffernan and Resnick
\cite{heffernanresnick2007} after slight changes in the assumptions
which make the theory more probabilistically viable.\

In \cite{heffernanresnick2007}, a bivariate random vector $(X,Y)$ is
considered, where the distribution of $Y$ is in the DOA of
an extreme value distribution $G_{\gamma}$, where, for $\gamma\in
\R$,
\begin{equation}\label{eqn:Ggamma}
G_{\gamma}(x)= \exp\{-(1+\gamma x)^{-1/\gamma}\},\qquad 1 + \gamma x
> 0 .
\end{equation}
For $\gamma= 0$, the distribution function is interpreted as
$G_{0}(x)= \exp\{-\mathrm{e}^{-x}\}, x \in\mathbb{R}$.
Instead of conditioning on $Y$ being large, their theory was developed
under the equivalent assumption of the existence of a vague limit for
the modified joint distribution of
a suitably scaled and centered $(X,Y)$. The precise definition
{(Definition~\ref{def:conevm}) is given in Section~\ref{subsec:modelsetup}
and} defines
{the \textit{conditional extreme value} (CEV) model. The CEV model
differs from classical MEVT and does not assume that the distribution
of $(X,Y)$ is in a multivariate DOA.
Only} one of the marginal distributions is assumed
to be in the univariate DOA of an extreme value
distribution.

{The CEV model is useful in two contexts. In the first, the MEV model
holds, but asymptotic independence makes it difficult to compute
probabilities of risk regions; in this case, the CEV model, if
applicable, provides a supplementary assumption to the MEV model
and thus may provide better risk estimates. Therefore, both the MEV and
CEV models are assumed to hold. This is the way in which hidden regular
variation may be used; see \cite{resnick2008} for some
background. In the other context, we do not assume that $(X,Y)$ is in a
multivariate domain of attraction and the CEV assumptions may still
hold; the CEV model is then a standalone model. In a study of Internet traffic data
\cite{dasresnick2009a,lopezoliverosresnick2009}, one variable was
found to be not in any
univariate DOA and hence MEVT was not applicable, but the conditional
model was still valid. }

In Section~\ref{sec:condandmult}, we complete the study of the
relationship between multivariate extreme value
theory and conditioned limit theory begun in \cite{heffernanresnick2007}.
The connection is through the
theory of regular variation on cones. The defining relation of MEVT
can be standardized to produce standard regular variation on the cone
$[0,\infty]^2\setminus\{(0,0)\}$. The limit relation in conditioned
limit theory can \textit{sometimes} be standardized to regular
variation on the smaller cone $[0,\infty]\times(0,\infty]$. We
explain the
precise circumstances when the CEV model can be standardized to
regular variation.

Section~\ref{sec:consistency} studies a consistency question for
conditional models related to one raised in
\cite{heffernantawn2004} {and its discussion} following the paper. In practice, for a
vector $(X,Y)$, one has a choice of whether to condition on $X$ being
large or $Y$ being large and, depending on the choice, different
models are potentially
possible. We show that if conditional approximations are
possible no matter which variable is chosen as the conditioning
variable, then, in fact, the joint distribution is in a classical
multivariate DOA of an extreme value law.
A related issue is when the CEV model
can be extended to a classical MEV model;
Section~\ref{sec:extension}
provides conditions for this.
Section~\ref{sec:CEVHRV} relates \textit{hidden regular
variation} \cite{resnick2002a,maulikresnick2005} and the CEV
model under the assumption of multivariate extreme value DOA
for $(X,Y)$ with asymptotic independence. Finally, Section~\ref
{sec:examples} presents some
examples in order to demonstrate features of the conditioned models
and the final section supplies some deferred proofs.\looseness=-1

\subsection{Notation}\label{subsec:glossary}

Below, we list some commonly used notation and provide some references.
\[
\begin{array}{llll}
\R_{+}^d            &\quad [0,\infty)^d. \mbox{ Also, similarly denote } \overline{\R}_{+}^d = [0,\infty]^d, \overline{\R}^d= [-\infty,\infty]^d. \\
\E^*                &\quad\mbox{A nice subset of the compactified finite-dimensional Euclidean space,}\\
                    &\quad\mbox{often denoted $\E$ with different subscripts and superscripts, as required.}
\\
\Ecal^*             &\quad\mbox{The Borel $\sigma$-field of the subspace $\E^*$. }\\
\mathbb{M}_+(\E^*)  &\quad\mbox{The class of Radon measures on Borel subsets of $\E^*$.}
\\
f^\leftarrow &\quad\mbox{The left-continuous inverse of a monotone function $f$. }\\
&\quad\mbox{For an increasing function, $\finv(x) = \inf\{y\dvtx  f(y) \ge x\}$.} \\
&\quad\mbox{For a decreasing function, $\finv(x) = \inf\{y\dvtx  f(y) \le x\}$.}
\\
RV_\rho&\quad \mbox{The class of regularly varying functions with index
$\rho$; see \cite
{seneta1976,binghamgoldieteugels1987,gelukdehaan1987,dehaanferreira2006,resnickbook2008}.}\\
[2mm]
\Pi&\quad \mbox{{The class of }$\Pi$-varying functions; see
\cite{binghamgoldieteugels1987,resnickbook2008}}.
\\
\E^{(\gamma)} &\quad \{x\dvtx 1+\gamma x>0\} \mbox{ for } \gamma\in\R.\\
\overline{\E}^{(\gamma)} &\quad \mbox{The closure on the right of the
interval
$\E^{(\gamma)}$.}\\
\overline{\overline{\E}}{}^{(\gamma)} &\quad \mbox{The closure on both
sides of the interval
$\E^{(\gamma)}$.}\\
\E^{(\lambda,\gamma)} &\quad \overline{\overline{\E}}{}^{(\lambda)} \times
\overline{\overline{\E}}{}^{(\gamma)}\setminus
\{(-\frac{1}{\lambda},-\frac{1}{\gamma})\}.\\
\E&\quad \mbox{Usually }[0,\infty]^2\setminus\{\bzero\}.\\
\E_0 &\quad \mbox{Usually }(0,\infty]^2.\\
\Enb&\quad [0,\infty]\times(0,\infty]. \mbox{ Similarly, } \Enl= (0,\infty
]\times[0,\infty].\\
\stackrel{v}{\to} &\quad \mbox{Vague convergence of measures; see \cite
{neveu1977,kallenberg1983}}.\\
G_\gamma&\quad \mbox{An extreme value distribution given by
\eqref{eqn:Ggamma} with parameter $\gamma\in\R$. }\\
D(G_\gamma) &\quad \mbox{The DOA of the extreme value
distribution $G_\gamma$; in other words, the set of
$F$'s }\\
{}&\quad\mbox{ satisfying
\eqref{eqn:doa}.
For
$\gamma>0$, $F \in D(G_\gamma)$
is
equivalent to $1-F \in RV_{-1/\gamma}$.}\\

\end{array}
\]
%

\subsection{Model setup and basic assumptions} \label{subsec:modelsetup}

Our model assumptions follow those of \cite{heffernanresnick2007}.
\begin{Definition}[(Conditional extreme value model)]\label{def:conevm}
Suppose that $(X,Y) \in\R^2$ is a random
vector and that there exist functions $\alpha(t) > 0, a(t) > 0$,
$\beta(t), b(t) \in\R$, a constant $\gamma\in\R$ and a non-null
Radon measure $\mu$ on Borel subsets of $[-\infty,\infty] \times
\overline{\E}^{(\gamma)} $ such that
\begin{eqnarray}
 (\mathrm{a})\quad t\P\biggl( \biggl(\frac{X- \beta(t)}{\alpha(t)}, \frac{Y - b(t)}{a(t)} \biggr)
\in \cdot \biggr) \cnvg\mu(\cdot) \qquad\mbox{in } \bM_{+}\bigl([-\infty,\infty]
\times
\overline{\E}^{(\gamma)} \bigr). \label{eqn:cond2}
\end{eqnarray}
Assume that $\mu$ satisfies the following \textit
{conditional non-degeneracy conditions}: for each $y \in\E^{\gamma}$,
\begin{eqnarray}
&& (\mathrm{b})\quad \mu\bigl([-\infty,x]\times(y,\infty]\bigr) \mbox{ is not a degenerate distribution in $x$}, \label{eqn:nondegx}\\
&& (\mathrm{c})\quad \mu\bigl([-\infty,x]\times(y,\infty]\bigr) < \infty.\label{eqn:munotinf}
\end{eqnarray}
Additionally, we assume that
\begin{eqnarray}
&& (\mathrm{d})\quad H(x):= \mu\bigl([-\infty,x]\times(0 ,\infty]\bigr) \mbox{ is a probability
distribution.} \label{eqn:Hispdf}
\end{eqnarray}
We say that $(X,Y)$ follows a \textit{conditional extreme value
model} (abbreviated CEV model) if conditions
\eqref{eqn:cond2}--\eqref{eqn:Hispdf} hold.
We write $(X,Y) \in\mathit{CEV}(\alpha,\beta,a,b,\gamma)$
and often ignore the parameters for generic usage.
\end{Definition}

Since convergence in \eqref{eqn:cond2} holds in $\bM_{+}([-\infty,\infty
] \times
\overline{\E}^{(\gamma)} )$, it also holds
without the $X$ variable and so if the marginal distribution of $Y$ is
$F$, then
$F \in D(G_{\gamma})$ for $\gamma\in
\R$, as defined in \eqref{eqn:Ggamma}; that is, as $t \to\infty$,
\begin{eqnarray}
t\bigl(1-F\bigl(a(t)y + b(t)\bigr)\bigr) = t\P \biggl(\frac{Y-b(t)}{a(t)} > y \biggr)
\to(1 + \gamma
y)^{-1/\gamma},\qquad 1+\gamma y > 0. \label{eqn:doa}
\end{eqnarray}
Also, conditions \eqref{eqn:cond2}, \eqref{eqn:nondegx} and
\eqref{eqn:munotinf} {imply that if $(x,0)$ is a continuity point of
$\mu(\cdot)$, then}
\begin{eqnarray}
\P\biggl(\frac{X-\beta(t)}{\alpha(t)} \le x \Big| Y > b(t) \biggr) \to H(x) =
\mu\bigl([-\infty,x]\times(0 ,\infty]\bigr)\qquad \mbox{as } t \to\infty,
\end{eqnarray}
that is, a conditioned limit holds. This accounts for the name \textit
{conditional
extreme value model} {and we can think of $Y$, the variable in a
univariate DOA, as the \textit{conditioning
variable}.}
Under the above assumptions, a
convergence to types argument \cite{heffernanresnick2007}
yields properties of the scaling and centering
functions: there exist functions $\psi_1,
\psi_2\dvtx  \R_{+} \mapsto\R$ such that for $c>0$,
\begin{eqnarray}
\lim_{t \to\infty} \frac{\alpha(tc)}{\alpha(t)} =
\psi_1(c), \qquad\lim_{t \to\infty}
\frac{\beta(tc)-\beta(t)}{\alpha(t)} = \psi_2(c).
\label{eqn:condpsi1psi2}
\end{eqnarray}
{This implies that $\psi_1(c) = c^{\rho}$ for some $\rho
\in\R$ (\cite{dehaanferreira2006}, Theorem B.1.3), and either
$\psi_2 \equiv0$ or $\psi_2(c) = k (c^{\rho}-1)/{\rho}$ for some
$k \neq
0$ (\cite{dehaanferreira2006}}, Theorem B.2.1).

\subsection{Comparison with the model proposed by Heffernan and Tawn}\label{subsec:ht}

The model discussed in \cite{heffernanresnick2007} was motivated by
ideas of Heffernan and Tawn \cite{heffernantawn2004}. The basic
premise is that in classical MEVT, probabilities of extreme sets
(values which are very high or very low) are calculated under the
existence of a joint extreme value limit. However, in
practice, we sometimes observe that only a subset of the components is
extreme or, alternatively, we are interested in regions
where all extreme values do not occur together. Let us look at a
description of Heffernan and Tawn's model with $d=2$ for simplicity
and ease of comparison with the formulation used in this paper.
\begin{enumerate}[(1)]
\item[(1)] Assume that $\bX=(X_1,X_2)$ is a random vector with joint
distribution $F_X$ and marginal distributions $F_1$ and $F_2$. Also,
assume that we have $n$ i.i.d.~copies of $\bX$.
\item[(2)] Assume that $C$ is an extreme set in the sense that, {for any
element in $C$, at least one of its} components is extreme. Define
\begin{eqnarray*}
C_i & = C \cap\{\bx\in\R^d\dvtx  F_{i}(x_i) > F_{j}(x_j)\},\qquad
i=1,2, j\neq i.
\end{eqnarray*}
Also, define $v_{X_i}=\inf_{\bx\in C_i}(x_i), i=1,2$, and assume
that {each}
$\P(X_i > v_{X_i})$ {is} close to $1$, making the $C_i$'s extreme and
hence making $C$ an extreme set. We then write
\begin{eqnarray*}
\P(\bX\in C)  = \P(\bX\in C_1) + \P(\bX\in C_2) = \sum_{i=1}^2 \P
(\bX\in C_i | X_i > v_{X_i}) \P(X_i > v_{X_i}).
\end{eqnarray*}
\item[(3)] The $X_i$'s are marginally assumed to be extreme-valued. A
generalized Pareto distribution is fitted to each of the marginals
above a threshold, in the case above, $v_{X_i}$; see
\cite{pickands1975}. Below the threshold, the marginals are
approximated by an empirical distribution. Denote the estimate
{of the marginal distributions by} $\hat{F_{i}}$.
\item[(4)] All the marginals are transformed to Gumbel marginals using the
transformation
\[
Y_i = -\log[-\log\{\hat{F_{i}}(x_i)\}], \qquad       i=1,2.
\]
\item[(5)] In order to estimate $\P(\bY\in C_i | {Y_i} > {v_{y_i}})$ (the
transformed case), a conditioned limit is assumed, as follows: there
exist normalizing vectors $a_1(y),a_2(y),b_1(y),b_2(y) \in\R$ such that
\begin{eqnarray}\label{eqn:condHT}
\lim_{y_i\to\infty}\P\bigl(Y_j \le a_j(y_i) y_j + b_j(y_i)| Y_i =y_i \bigr) =G_i(y_j),\qquad
i=1,2, y_j \in\R.
\end{eqnarray}
\item [(6)] The parameters $a_i(y)$ and $b_i(y)$ are estimated by assuming a
parametric structural form; see \cite{heffernantawn2004} for details.
\end{enumerate}
{For} the model defined in \cite{heffernanresnick2007} {and} used in
this paper:
\begin{enumerate}[(6)]
\item[(1)] We start with the same assumption (1).
\item[(2)] We {focus} on one of the extreme sets $C_1$ and $C_2$; without
loss of generality, assume this {is}
$C_2$. {We} assume only
one of the marginals $X_2$ is extreme-valued.
In \cite{heffernantawn2004}, all of the marginals are extreme-valued.
\item[(3)] Instead of fitting an exact GPD over a threshold for the marginal
distribution, we assume that $X_2 \in D(G)$, in the sense of \eqref{eqn:doa}.
\item[(4)] Instead of Gumbel marginals, we transform to Pareto marginals
for $X_2$, which {facilitates
the use of} tools from standard
regular variation theory. Thus, the transformation here is
$X_2^{*}=1/(1-F(X_2))$ and $X_1$ remains unchanged.
\item[(5)] In order to estimate $\P((X_1,X_2^*) \in C_2 | X_2^* >
x_2)$, a conditioned limit is assumed: there exist
normalizations
$\alpha(x_2)>0, \beta(x_2)\in\R$ such that
\begin{eqnarray}\label{eqn:condHefRes}
\lim_{x_2\to\infty} \P\bigl(X_1 \le\alpha(x_2) x_1 + \beta(x_2)| X_2^*
> x_2 \bigr) = G(x_1),\qquad x_1 \in\R.
\end{eqnarray}
This is equivalent to \eqref{eqn:cond2} when $Y=X_2^*$ has been
standardized to a Pareto margin.
\end{enumerate}

A technique for estimating model parameters has been discussed in
\cite{fougeressoulier2009a}. Further, if one makes precise what
version of the conditional distribution is being used in \eqref{eqn:condHT},
then, as expected, it is shown in \cite{resnickzeber2010} that
\eqref{eqn:condHT} implies \eqref{eqn:condHefRes}, but \eqref
{eqn:condHefRes} may
hold without \eqref{eqn:condHT} holding.

\section{A consistency result for conditional extreme value models}
\label{sec:consistency}

The CEV model {defined} in Section~\ref{subsec:modelsetup} is
not symmetric {in $X$ and $Y$}. So, given bivariate data,
which component should serve as the conditioning variable?
A similar issue was raised in
\cite{heffernantawn2004} and
\cite{heffernanresnick2007}. Heffernan and
Tawn \cite{heffernantawn2004} {considered $(X,Y)$ in a multivariate
DOA with
asymptotic independence, introduced the supplementary assumption that a
conditional model was also valid and raised the question of criteria
for deciding which
variable to make the conditioning variable.}
{If either variable could be made the conditioning variable, then they
considered self-consistency of the two conditional models.
Assuming densities, they provided a natural constraint of equality of
joint limiting
densities under each model for the common region where both
models were defined. We consider a related problem without assuming that
$(X,Y)$ has a distribution in a multivariate domain.}

{Definition~\ref{def:conevm} does not assume that the distribution of
$(X,Y)$ is in a multivariate
DOA.} Suppose that
$X \sim F_X, Y\sim F_Y$. Assume that {$(X,Y) \in
\mathit{CEV}(\alpha,\beta,a,b,\gamma)$} with limit measure
$\mu_{X,Y>}(\cdot)$
and $F_Y \in D(G_{{\gamma}})$, and {also}
$(Y,X) \in \mathit{CEV}(c,d,\chi,\phi,\lambda)$ with limit measure $\mu
_{Y,X>}(\cdot)$ and $F_X \in
D(G_{\lambda})$. Assuming
{both}
conditional models {implies} that $(X,Y)$ is in the DOA
of a bivariate extreme value distribution $G$. If the
limit distribution $G$ is not a product measure,
then
$\mu_{X,Y>}$ and $\mu_{Y,X>}$ are equal up to linear transformation on
subsets that are defined on the intersection of the domains of both
measures. {Recall that if, marginally, $F_X \in
D(G_{\lambda})$ and $F_Y \in D(G_{\gamma})$, then we do not necessarily
have $(X,Y) \in D(G)$ for a bivariate extreme value distribution $G$;
see \cite{schlather2001}.}
{The precise consistency statement is next; the proof is deferred to
Section~\ref{sec:proofs}.}
%
\begin{Theorem}\label{thm:gencon}
Suppose we have a bivariate random vector $(X,Y) \in\R^2$,
non-negative functions $\alpha(\cdot), a(\cdot),$ $ \chi(\cdot), c(\cdot)$
and real-valued functions $
\beta(\cdot), b(\cdot) ,
\phi(\cdot),d(\cdot) $ such that $(X,Y)\in \mathit{CEV}(\alpha,\beta,a,b,\gamma
)$, that is,
\begin{eqnarray}
t\P\biggl[ \biggl(\frac{X-\beta(t)}{\alpha(t)},
\frac{Y-b(t)}{a(t)} \biggr) \in \cdot \biggr] & \cnvg
\mu_{X,Y>}(\cdot) \qquad\mbox{in } \mathbb{M}_{+} \bigl( [-\infty,\infty] \times
\overline{\E}^{(\gamma)} \bigr) \label{eqn:c1}
\end{eqnarray}
and $(Y,X)\in \mathit{CEV}(c,d,\chi,\phi,\lambda)$, that is,
\begin{eqnarray}
t\P\biggl( \biggl(\frac{X-\phi(t)}{\chi(t)}, \frac{Y-d(t)}{c(t)} \biggr)
\in \cdot \biggr) & \cnvg\mu_{Y,X>}(\cdot) \qquad\mbox{in } \mathbb{M}_{+} \bigl(
\overline{\E}^{(\lambda)} \times[-\infty,\infty] \bigr)
\label{eqn:c2}
\end{eqnarray}
for $\lambda,\gamma\in\R$, where both $\mu_{X,Y>} $ and $\mu_{Y,X>} $
satisfy conditional non-degeneracy conditions
corresponding to \eqref{eqn:nondegx} and \eqref{eqn:munotinf}.
Then $(X,Y)$ is in the DOA of a multivariate
extreme value distribution on $\E^{(\lambda,\gamma)}$, that is,
%
\begin{equation}\label{eqn:multDOA}
t\P\biggl( \biggl(\frac{X-\phi(t)}{\chi(t)},
\frac{Y-b(t)}{a(t)} \biggr) \in \cdot \biggr) \cnvg
\mu_{X,Y}(\cdot)\qquad \mbox{in } \mathbb{M}_{+} \bigl(\E^{(\lambda,\gamma)} \bigr),
\end{equation}
where $\mu_{X,Y}(\cdot) $ is a
non-null Radon measure on $ \E^{(\lambda,\gamma)}$.
\end{Theorem}
%
\begin{Remark}
Theorem~\ref{thm:gencon} does not impose a restriction on the scaling
and centering functions of $X$ and $Y$, {which means} that the joint
conditional convergences \eqref{eqn:c1} and \eqref{eqn:c2} impose
sufficient regularity so that $(X,Y)$ belongs to a joint
DOA. Equation \eqref{eqn:multDOA} says that $\mu_{X,Y}$ is the exponent
measure of an extreme value distribution $G$. The following are further
consequences from the proof of Theorem~\ref{thm:gencon}:
\begin{enumerate}[(1)]
\item[(1)] if $(X,Y)$ is not
asymptotically independent, then we get $\alpha\sim k_1 \chi$ and $c
\sim k_2 a$ for some non-zero constants $k_1,k_2$, hence
$\mu_{X,Y>}$ and $\mu_{Y,X>}$ are equal up to linear transformations;
\item[(2)] if $(X,Y)$ is asymptotically independent, then $\lim_{t \to\infty
}{\alpha(t)}/{\chi(t)} = 0$,
$\lim_{t \to\infty}{c(t)}/\break{a(t)} = 0$.
\end{enumerate}

\end{Remark}

\textit{Consistency: Standard regularly varying case}.
We were led to Theorem~\ref{thm:gencon} by considering the
special case of {standard}
regular variation where
$(X,Y)$ satisfies
$(X,Y) \in\mathit{CEV}(\alpha(t)=t,\beta(t)=0,a(t)=t,b(t)=0,\gamma=1),$
$(Y,X) \in\mathit{CEV}(\alpha(t)=t,\beta(t)=0,a(t)=t,b(t)=0,\gamma=1)$
and the vague convergence in \eqref{eqn:cond2} is regular variation
on the cone $\Enb= [0,\infty] \times(0,\infty]$
(\cite{davydovmolchanovzuyev2007,resnick2008}, \cite{resnickbook2007}, page~173). We can show
\cite{das2009} that if
\begin{eqnarray}
t\P[ t^{-1}( {X}, {Y}) \in \cdot ]
& \cnvg&\mu_{X,Y>}(\cdot)\qquad \mbox{in } \mathbb{M}_{+} ( \Enb
),
\label{eqn:tomuinEnb}\\
t\P[ t^{-1}( {X}, {Y}) \in \cdot ]
& \cnvg&\mu_{Y,X>}(\cdot)\qquad \mbox{in } \mathbb{M}_{+} ( \Enl
),\label{eqn:tonuinEnl}
\end{eqnarray}
where $\mu_{X,Y>} $ and $\mu_{Y,X>} $
satisfy the conditional non-degeneracy conditions
\eqref{eqn:nondegx} and \eqref{eqn:munotinf}, then
$(X,Y)$ is standard regularly varying on $\E:=[0,\infty]^2\setminus\{
\bzero\}$, that is,
\begin{eqnarray}
t\P\biggl[ \biggl(\frac{X}{t}, \frac{Y}{t} \biggr) \in \cdot \biggr]
 \cnvg\mu_{X,Y}(\cdot)\qquad \mbox{in } \mathbb{M}_{+} ( \E ),
\label{eqn:tomunuinE}
\end{eqnarray}
where $\mu_{X,Y}$ is a Radon measure on $\E$ such that
\begin{eqnarray*}
\mu_{X,Y}|_{\Enb}(\cdot) = \mu_{X,Y>}(\cdot) \qquad\mbox{on } \Enb
\quad\mbox{and}\quad
\mu_{X,Y}|_{\Enl}(\cdot) = \mu_{Y,X>}(\cdot)\qquad \mbox{on } \Enl.
\end{eqnarray*}

{A proof and discussion of} the absolutely continuous case
is in \cite{das2009}.
\begin{Example} \label{ex:ex1}
Suppose that $(X,Y)$ is a bivariate random variable with joint density
\[
f_{X,Y} (x,y) = \frac{4x}{(x^2+y)^3} + \frac{4y}{(x+y^2)^3},\qquad
x \ge1, y\ge1.
\]
The following hold as $t \to\infty$:
\begin{eqnarray*}
t^2f_X(tx) &\to&\frac{2}{x^2},\qquad t^2f_Y(ty) \to\frac{2}{y^2},\qquad x,y >0,\\
t^{5/2}f_{X,Y}\bigl(tx,\sqrt{t}y\bigr) &\to&\frac{4y}{(x+y^2)^3} =: g_1(x,y) \in L_1(\Enl),\\
t^{5/2}f_{X,Y}\bigl(\sqrt{t}x,ty\bigr) &\to&\frac{4x}{(x^2+y)^3} =: g_1(x,y) \in L_1(\Enb).
\end{eqnarray*}
This means that the conditions of Theorem~\ref{thm:gencon} hold. Note
that we here have identical Pareto marginals. We are thus led to the
analogs of \eqref{eqn:cond2} on
different cones:
\begin{eqnarray*}
 t\P\biggl(\frac{X}{\sqrt{t}} \le x, \frac{Y}{t} > y\biggr) &\to&\frac{1}{y} -
\frac{1}{y + x^2},\qquad x \ge0, y > 0,\\
 t\P\biggl(\frac{X}{t} > x, \frac{Y}{\sqrt{t}}\le y\biggr) &\to&\frac{1}{x} - \frac
{1}{x + y^2},\qquad x >0, y \ge0,\\
 t\P\biggl( \biggl(\frac{X}{t}, \frac{Y}{t} \biggr) \in([0,x]\times[0,y])^c \biggr) &\to&
\frac{1}{x} + \frac1y,\qquad x>0, y>0.
\end{eqnarray*}
\end{Example}

\section{The CEV model and standard regular variation} \label{sec:condandmult}

As remarked after Theorem~\ref{thm:gencon},
questions about the general conditional model are effectively
analyzed by starting with standard regular variation on the
cones $\Enl$ or $\Enb$. It is {theoretically} useful to know when
standardization of the conditional extreme value model is
possible. A partial answer appears in
\cite{heffernanresnick2007}, Section 2.4,
and we consider this issue in more detail, {starting with a
review and definition of \textit{standardization}
\cite{dehaanferreira2006,dehaanresnick1977,resnickbook2007,resnickbook2008}.}

\subsection{Standardization}

Standardization is the process of marginally transforming a random
vector $\bX$ into a different vector $\bZ^*$,
$\bX\mapsto\bZ^*,$
so that the distribution of $\bZ^*$ is standard regularly varying
on a cone $\E^*$; that is,
for some Radon measure $\mu^{*}(\cdot)$,
%
\begin{equation}\label{eqn:standardizeTheSucker}
t\P[t^{-1} {\bZ^*} \in \cdot ]
\stackrel{v}{\to} \mu^{*}(\cdot)\qquad \mbox{in }\mathbb{M}_+(\E^*
).
\end{equation}
Depending on the cone, one or more components of $\bZ^*$
are asymptotically Pareto. For classical multivariate
extreme value theory, each component is asymptotically Pareto and
$\E^* =\E=[\bzero,\binfty]\setminus\{\bzero\}.$ The
technique is used in classical multivariate extreme value theory to
characterize multivariate domains of attraction and dates back to at least
\cite{dehaanresnick1977}; see also
\cite{mikosch2005,mikosch2006,dehaanferreira2006,resnickbook2007}
and \cite{resnickbook2008}, Chapter 5.
Standardization is analogous
to the copula transformation, but is
better suited to studying limit relations
\cite{kluppelbergresnick2008}.

In Cartesian coordinates, the limit measure in
\eqref{eqn:standardizeTheSucker}
has the scaling property
%
\begin{equation}\label{eqn:muscales}
\mu^{*}(c \cdot)=c^{-1} \mu^{*}(\cdot),\qquad c>0.
\end{equation}
This scaling in Cartesian coordinates translates to a product limit
when expressed in
polar coordinates. An angular measure
exists, allowing
the characterization of limits
\[
\mu^{*}\biggl\{\bx\dvtx  \Vert\bx\Vert>r,\frac{\bx}{\Vert\bx\Vert} \in\Lambda\biggr\}
=r^{-1} S(\Lambda)
\]
for Borel subsets $\Lambda$ of the unit sphere in $\E^*$.

In classical multivariate extreme value theory, $S$ is a finite
measure which we may take to be a probability measure without loss
of generality. However, when $\E^*=\Enb$,
$S$ is \textit{not} necessarily finite because absence of the
horizontal axis boundary in $\Enb$
implies the unit sphere is not compact.

Here is an explicit description of \textit{standardization}.
Suppose that $\bX= (X_1,X_2,\ldots,X_d)$ is a random vector in $\R^d$
which satisfies
\begin{eqnarray}\label{eqn:mvev}
t\P\biggl[ \biggl(\frac{X_1-\beta_1(t)}{\alpha_1(t)},\frac{X_2-\beta_2(t)}{\alpha
_2(t)},\ldots,\frac{X_d-\beta_d(t)}{\alpha_d(t)} \biggr)
\in \cdot \biggr] \cnvg\mu(\cdot)\qquad \mbox{in }
\mathbb{M}_{+}(\D)
\end{eqnarray}
for some $\D\subset\overline{\R}^d$, $\alpha_i (t)>0, \beta_i(t)\in
\mathbb{R}$ for $i=1,\ldots,d$. Suppose that we have $\boldf=
(f_1,\ldots,f_d)$ such that, for $i=1,\ldots,d$:
\begin{longlist}[(a)]
\item[(a)] $f_i\dvtx  \mbox{range of } X_i \to(0,\infty)$;
\item[(b)] $f_i$ is monotone;
\item[(c)] $\nexists K >0$ such that $|f_i| \le K$.
\end{longlist}
Then $\boldf$ \textit{standardizes} $\bX$ if $\bZ^*=\boldf
(\bX)=(f_i(X_i), i=1,\ldots,d)$ satisfies \eqref{eqn:standardizeTheSucker}.
We call $\boldf$ the \textit{standardizing function} and say
\eqref{eqn:standardizeTheSucker} is the \textit{standardization} of
\eqref{eqn:mvev}.

For the conditional model defined in Definition~\ref{def:conevm} in
Section~\ref{subsec:modelsetup},
where $F$, the distribution of $Y$, satisfies $F \in
D(G_\gamma)$, we can
always use
$b(\cdot) =({1}/(1-F))^\leftarrow
(\cdot)$ to standardize $Y$ and
$Y^{*} =b^{\leftarrow}(Y)$
is the standardization of $Y$; see \cite{heffernanresnick2007}.

\subsection{When can the conditional extreme value model be
standardized?}\label{subsec:when}

Suppose that $(X,Y)$ satisfies Definition~\ref{def:conevm} and, in
particular, \eqref{eqn:cond2} holds.
Standardization in \eqref{eqn:cond2}
is possible \textit{unless} $(\psi_1,\psi_2)= (1,0)$, which is
equivalent to the limit measure being a product measure
\cite{heffernanresnick2007}.
The converse is also true. Consequently, when the limit measure is not a
product measure, we can reduce to standard regular variation
on the cone $\Enb$ and, conversely, we can think of the general
conditional model as a transformation of standard regular variation
on $\Enb$.

We begin by showing that when we have standardized convergence on $\E
_\sqcap$,
the limit measure cannot be a product measure.
\begin{Lemma}\label{lem:stdnotprod}
Suppose that $(X,Y)$ is standard regularly varying on the cone $\Enb$
such that
\begin{eqnarray}\label{eqn:std}
t\P[ t^{-1}({X}, {Y}) \in \cdot ]
\cnvg\mu(\cdot) \qquad\mbox{in } \mathbb{M}_{+} ( \Enb
)
\end{eqnarray}
for some non-null Radon measure $\mu(\cdot)$ on $\Enb$
satisfying the conditional non-degeneracy conditions as in
\eqref{eqn:nondegx} and
\eqref{eqn:munotinf}.
Then $\mu(\cdot)$ cannot be a product measure.
\end{Lemma}

\begin{pf} If $\mu$ is a product measure, then we have
\begin{eqnarray}
\mu\bigl([0,x]\times(y,\infty] \bigr)  = G(x)y^{-1}\qquad \mbox{for } x\ge
0, y >0 \label{eqn:prod}
\end{eqnarray}
for some finite distribution function $G$ on $[0,\infty)$.
Now, \eqref{eqn:std} implies that $\mu$ is homogeneous
of order $-1$, that is,
\begin{eqnarray}
\mu(c\Lambda) = c^{-1} \mu(\Lambda)
\qquad\forall c
> 0, \label{eqn:stdhom}
\end{eqnarray}
where $\Lambda$ is a Borel subset of $\Enb$. Therefore, using \eqref{eqn:prod},
\begin{eqnarray*}
\mu\bigl(c([0,x]\times(y,\infty]) \bigr) = \mu\bigl([0,cx]\times(cy,\infty] \bigr)
= G(cx) {(cy)}^{-1} = c^{-1} G(cx) y^{-1}.
\end{eqnarray*}
Moreover, using \eqref{eqn:prod} and
\eqref{eqn:stdhom}, $\mu(c([0,x]\times(y,\infty]) ) =
c^{-1}G(x)y^{-1}$
and, therefore,
$ G(cx) = G(x)$ $\forall c > 0, x > 0.$
Hence, for fixed $ y \in\E^{(\gamma)}, c>0, x>0$,
\begin{eqnarray*}
\mu\bigl([0,cx]\times(y,\infty] \bigr) = G(cx)y^{-1} & = G(x)y^{-1} = \mu\bigl([0,x]\times(y,\infty] \bigr).
\end{eqnarray*}
Thus, $\mu$ becomes a degenerate distribution in $x$,
contradicting our conditional non-degeneracy assumptions and, consequently,
$\mu(\cdot)$ cannot be a product measure.
\end{pf}

Suppose we have a general CEV model as
in Definition~\ref{def:conevm} with product limit measure.
We show this CEV model cannot be \textit{standardized} to
regular variation on some cone $\C\subset\E$ ($\C= \Enb$
in our case). Since Definition~\ref{def:conevm} implies that
$Y$ can always be standardized, in the following, we assume that $Y^*$ is
the standardized version of $Y$ and we only consider the problem of
standardizing $X$.
\begin{Theorem}\label{thm:productsBad}
Suppose that $X \in\R, Y^{*} > 0$ are random variables such that for
functions $\alpha(\cdot) > 0, \beta(\cdot) \in\R$, we have, as $t \to
\infty$,
\begin{eqnarray}
t\P\biggl[ \biggl(\frac{X-\beta(t)}{\alpha(t)},
\frac{Y^{*}}{t} \biggr) \in \cdot \biggr]  \cnvg
G \times\nu_1(\cdot)\qquad \mbox{in } \mathbb{M}_{+} \bigl(
[-\infty,\infty] \times(0,\infty]
\bigr),\label{eqn:nonstdind}
\end{eqnarray}
where $\nu_1(x,\infty] =
x^{-1} $, $x>0$, and $G$ is some finite, non-degenerate distribution on
$\R$.
Then there does not exist a standardizing function,
$f(\cdot)\dvtx \mbox{range of } X \mapsto(0, \infty)$, in the sense of the
discussion
after \eqref{eqn:mvev}, such that
%
\begin{equation}
t\P[t^{-1}({f(X)}, {Y^{*}}) \in \cdot ] \cnvg\mu(\cdot)\qquad \mbox{in }
\mathbb{M}_{+} ( \Enb
), \label{eqn:fstd}
\end{equation}
where $\mu$ satisfies the conditional non-degeneracy conditions.
\end{Theorem}

\begin{pf}
Note that $Y^{*}$ is already standardized here. Suppose that there
exists a standardization function $f(\cdot) $ such that \eqref
{eqn:fstd} holds. Without loss of generality,
assume $f(\cdot)$ to be non-decreasing. This implies that for $\mu
$-continuity points $(x,y)$, we have
\begin{eqnarray*}
 t\P[{f(X)} \le{t} x, {Y^{*}}> {t}y] \to
\mu\bigl((-\infty,x]\times(y,\infty] \bigr)       \qquad           (t\to\infty),
\end{eqnarray*}
which is equivalent to
\begin{eqnarray}\label{eqn:muxy}
\hspace*{-10pt} t\P
\biggl(\frac{X-\beta(t)}{\alpha(t)} \le
\frac{f^{\leftarrow}(xt)-\beta(t)}{\alpha(t)}, \frac{Y^{*}}{t}>
y \biggr] \to\mu\bigl((-\infty,x]\times(y,\infty] \bigr)       \qquad       (t\to\infty).
\end{eqnarray}
Since $\mu((-\infty,x]\times(y,\infty]) < \infty$ and is
non-degenerate
in $x$, we have, as $t \to\infty$, that
\begin{eqnarray}\label{eqn:finvg}
\bigl(f^{\leftarrow}(xt)-\beta(t) \bigr)/\alpha(t) &\to h(x)
\end{eqnarray}
for some non-decreasing function $ h(\cdot)$ which has at least two
points of increase. Thus, \eqref{eqn:muxy} and \eqref{eqn:finvg} imply that
$\mu((-\infty,x]\times(y,\infty]) = G(h(x)) \times y^{-1}.$
Hence, $\mu(\cdot)$ is a product measure which, by Lemma
\ref{lem:stdnotprod}, is not possible.
\end{pf}

\textit{Summary}.
{There follows a summary} {describing} when standardization is
possible and the relationship of standardization to the limit measure
being a
product. Part 2 is proved in Section
\ref{sec:proofs}.
Statistical methods for detecting when a CEV model is appropriate and
whether the limit measure is a product are given in
\cite{dasresnick2009a}:

\begin{enumerate}[(1)]
\item[(1)] Suppose that $(X,Y)$ satisfy Definition
\ref{def:conevm} so that the limits in \eqref{eqn:condpsi1psi2} hold.
If $(\psi_1,\psi_2) \ne(1,0)$,
then there exists a standardiz{ing} function $\boldf= (f_1,f_2)$
such that $(X^*,Y^*)=(f_1(X),f_2(Y))$ is standard regularly varying on
$\Enb$,
\[
t\P[ t^{-1}
({f_1(X)}, {f_2(Y)}) \in \cdot ] \cnvg\mu^{**}(\cdot)
\qquad\mbox{in } \mathbb{M}_{+} ( \Enb
)
\]
and $\mu^{**}$ is a non-null Radon measure satisfying the conditional
non-degeneracy conditions.

\item[(2)] Conversely, suppose that we have a bivariate random vector $(X^*,Y^*)
\in\R^2_{+}$ satisfying
\[
t\P[t^{-1}({X^*}, {Y^*}) \in \cdot ]
\cnvg\mu^{**}(\cdot) \qquad\mbox{in } \mathbb{M}_{+} ( \Enb
),
\]
where $\mu^{**}$ is a non-null Radon measure satisfying the
conditional non-degeneracy conditions. Consider functions
$\alpha(\cdot) > 0 , \beta(\cdot) \in
\R$ such that \eqref{eqn:condpsi1psi2} holds with
$(\psi_1,\psi_2) \ne(1,0)$. There then
exist functions $a(\cdot)
> 0, b(\cdot) \in\mathbb{R} $ satisfying
\eqref{eqn:doa} and
$ \lambda(\cdot)\in\R$, $\gamma\in\mathbb{R}$ such that
\begin{eqnarray}
\hspace*{-29pt}t\P\biggl[ \biggl(\frac{\lambda(X^*)-\beta(t)}{\alpha(t)},
\frac{b(Y^*)-b(t)}{a(t)} \biggr) \in \cdot \biggr]  \cnvg
\tilde{\mu}(\cdot) \qquad\mbox{in } \mathbb{M}_{+} \bigl( [-\infty,\infty] \times
\overline{\E}^{(\gamma)} \bigr), \label{eqn:trans}
\end{eqnarray}
where $\tmu$ is a non-null Radon measure in
$[-\infty,\infty] \times\overline{\E}^{(\gamma)}
$ satisfying the conditional non-degeneracy conditions and $b(Y^*) \in
D(G_{\gamma})$.
\end{enumerate}
\begin{Remark}\label{rem:standardizationLives}
The previous summary applies to attempts to produce a standard pair by
marginal transformations. If one waives the requirement that only
marginal transformations be used, more is possible.
Suppose that $H$ is a non-degenerate probability and, in
$\mathbb{M}_{+} ([-\infty,\infty]\times(0,\infty])$,
\begin{eqnarray*}
t\P\biggl[ \biggl( \frac{X-\beta(t)}{\alpha(t)} , \frac{Y^{*}}{t}
\biggr) \in \cdot \biggr] \cnvg H \times
\nu_1(\cdot).
\end{eqnarray*}
Define
$X^{*} = ({(X-\beta(Y^{*}))Y^{*}} )/{\alpha(Y^{*})}.$
Then \cite{heffernanresnick2007} in $\mathbb{M}_{+} ([-\infty,\infty
]\times(0,\infty])$,
\[
t\P\biggl( \frac{X^{*}}{t} \le x , \frac{Y^{*}}{t}> y \biggr) \to\int_{0}^{1/y} H(xv)
\,\mathrm{d}v=\frac1x \int_0^{x/y} H(s)\,\mathrm{d}s \qquad (t\to\infty).
\]
The limit measure is homogeneous of order $-1$ and, thus, a
transformation of $(X,Y^*)$ to a standard regularly varying pair
exists, even when we have a limit
measure which is a product. Note that this transformation
{ is more complex than just a marginal
transformation and
is not in
the sense of the discussion after \eqref{eqn:mvev}.}
\end{Remark}

\subsection{A characterization of regular variation on $\E_\sqcap$}
The CEV model with limit measure which is not a product can always be
standardized to give regular variation on $\Enb$, so we would like
useful characterizations of such regular variation.
Standard regular variation on $\E$ was characterized by
\cite{dehaan1978} in terms of one-dimensional regular variation of
max-linear combinations and
\cite{resnick2002a} provides a characterization of hidden regular
variation in $\E$ and $\E_0$ in terms of max- and min-linear
combinations of the random vector. The following are comparable results for
$\Enb$.
\begin{Proposition} \label{prop:charEsqcap}
Suppose that $(X,Y) \in{\R_+^2}$ is a random vector and {$X>0$}
almost surely. The following are equivalent:
\begin{enumerate}[(1)]
\item[(1)]$(X,Y) $ is standard multivariate
regularly varying on $\Enb$ with limit measure satisfying the
non-degeneracy conditions \eqref{eqn:nondegx} and \eqref{eqn:munotinf};
\item[(2)] for all $a \in(0,\infty]$, we
have
%
\begin{equation}\label{eqn:charchar}
\lim_{t \to\infty} t \P\bigl( t^{-1} {\min(aX,Y)} > y \bigr) =c(a) y^{-1},\qquad y>0,
\end{equation}
for some non-constant, non-decreasing function $c\dvtx  (0,\infty] \to
(0, \infty)$.
\end{enumerate}
\end{Proposition}

\begin{pf}
$(2) \Rightarrow(1)$:
Assume that \eqref{eqn:charchar} holds
for some function $c\dvtx  (0,\infty] \to(0, \infty)$. Then, for $x \ge
0, y > 0$,
\begin{eqnarray*}
t\P\biggl( \frac{X}{t} \le x , \frac{Y}{t} > y \biggr)
& =& t\P\biggl(\frac{Y}{t} > y \biggr) - t\P\biggl( \frac{X}{t} > x , \frac{Y}{t} > y\biggr) \\
& =& t\P( X > 0,{Y} > ty ) - t\P\bigl( {(y/x)X} > ty , {Y} >ty\bigr) \\
& =& t\P\bigl({\min(a_1X,Y)} > ty \bigr) - t\P\bigl({\min\bigl((y/x)X,Y\bigr)} >t y \bigr) \qquad (a_1 :=\infty)\\
& \to& c(\infty) y^{-1} - c(y/x) y^{-1} \qquad (t \to\infty)\\
& =:&\mu([0,x]\times(y,\infty]).
\end{eqnarray*}
Since $c(\cdot)$ is non-decreasing and non-constant, $\mu$ is a
non-null Radon measure on $\Enb$ and we have our result. The
non-degeneracy of $\mu$ follows from the fact that $c(\cdot)$ is a
non-constant function.

$(1) \Rightarrow(2)$: Assume now that $(X,Y)$ is standard
multivariate regularly varying on $\Enb$. Hence, there exists a
non-degenerate Radon measure $\mu$ on $\Enb$ such that
\begin{eqnarray*}
\lim_{t \to\infty} t\P\biggl( \frac{X}{t} \le x , \frac{Y}{t} > y \biggr)  = \mu\bigl([0,x]\times(y,\infty]\bigr)
\end{eqnarray*}
and, for any $a \in(0, \infty]$,
\begin{eqnarray*}
t\P\biggl( \frac{\min(aX,Y)}{t} > y \biggr)
& =& t\P\biggl( \frac{X}{t} >\frac{y}{a} , \frac{Y}{t} > y \biggr)\to\mu\biggl(\biggl(\frac{y}{a},\infty\biggr]\times(y,\infty]\biggr)\qquad (t \to\infty)\\
& =& y^{-1} \mu\biggl(\biggl(\frac{1}{a}, \infty\biggr]\times(1,\infty]\biggr) =: c(a)y^{-1},
\end{eqnarray*}
by defining $c(a) = \mu((a^{-1}, \infty]\times(1,\infty])$ and
using the homogeneity property \eqref{eqn:muscales}. Note that the
conditional non-degeneracy of $\mu$ implies that $c$ is non-constant
and non-decreasing.
\end{pf}

The condition ``$X > 0$ almost surely'' in Proposition~\ref
{prop:charEsqcap} can be removed if we assume that $\lim_{t \to\infty} t\P({Y} > t) \to1$.

\subsection{Polar coordinates} \label{subsec:polar}

Section~\ref{subsec:when} shows that when the limit measure is not
a product measure, we can transform $(X,Y)$ to $(X^*,Y^*)$ such that
\begin{eqnarray}\label{eqn:stdstar}
\P[t^{-1}({X^*},{Y^*}) \in \cdot ]
\cnvg\mu^{**}(\cdot)\qquad \mbox{in } \mathbb{M}_{+}(\Enb).
\end{eqnarray}
Hence, $\mu^{**}$ satisfies \eqref{eqn:muscales}
and, when written in polar coordinates, has a spectral
form \cite{heffernanresnick2007}, Section 3.2. We
summarize some useful facts.
For convenience, take the norm
$\|(x,y)\|=|x|+|y|, (x,y)\in\R^2,$
although any other norm would suffice. A
standard homogeneity argument \cite{resnickbook2008}, Chapter
5, yields, for $r>0$ and $\Lambda$ a Borel subset of $[0,1)$,
\begin{eqnarray}
&&\mu^{**} \biggl\{ (x,y) \in[0,\infty]\times(0,\infty]\dvtx  x+y>r,
\frac{x}{x+y} \in\Lambda
\biggr\} \nonumber
\\[-8pt]\\[-8pt] \nonumber
&&\quad= r^{-1}\mu^{**} \biggl\{ (x,y)\in[0,\infty]\times(0,\infty]\dvtx x+y>1,
\frac{x}{x+y} \in\Lambda
\biggr\} =: r^{-1} S(\Lambda), \nonumber\label{eqn:defS}
\end{eqnarray}
where $S$ is a Radon measure on $[0,1)$.
For $x>0, y>0$, we get, from \eqref{eqn:defS},
\begin{eqnarray}
\mu^{**} \bigl([0,x]\times(y,\infty] \bigr) & = y^{-1}\int_0^{x/(x+y)}
(1-w)S(\mathrm{d}w) -x^{-1}\int_0^{x/(x+y)}wS(\mathrm{d}w).\label{eqn:altForm}
\end{eqnarray}
$S$ need not be a finite measure on $[0,1)$, but to guarantee that
%
\begin{equation}\label{eqn:defHStar}
H^{**}(x):=\mu^{**} \bigl( [0,x]\times(1,\infty] \bigr)
\end{equation}
is a probability measure, we can see by taking $x \to\infty$ in \eqref
{eqn:altForm} that we
need
%
\begin{equation}\label{eqn:condOnS}
\int_0^1(1-w)S(\mathrm{d}w)=1.
\end{equation}

Conclusion: The {class of conditional limits
$H^{**}(x)=\lim_{t \to\infty}P[{X^*}/{t} \leq
x|Y^*>t]$ or limits } $\mu^{**}$ in \eqref{eqn:stdstar}
is indexed by Radon measures $S$ on $[0,1)$ satisfying
condition \eqref{eqn:condOnS}.
\begin{Example}[(Finite angular measure)] \label{ex:finang}
If $S$ is uniform on $[0,1)$, $S(\mathrm{d}w) = 2\,\mathrm{d}w$, then
\eqref{eqn:condOnS} is satisfied and we have
\[
\mu^{**}\bigl([0,x]\times(y,\infty]\bigr) = \frac{x}{y(x+y)}.
\]
Putting $y=1$, we get {the Pareto distribution}
$H^{**}(x)=1-(1+x)^{-1} $ for $ x>0.$
\end{Example}
\begin{Example}[(Infinite angular measure)]\label{ex:infang}
The infinite measure $S(\mathrm{d}w) = (1-w)^{-1}\,\mathrm{d}w$
satisfies equation \eqref{eqn:condOnS} and we have
\[
\mu^{**}\bigl([0,x]\times(y,\infty]\bigr) = \frac{1}{y} + \frac{1}{x}
\log\biggl(1-\frac{x}{x+y}\biggr).
\]
Putting $y=1$ yields
$H^{**}(x) =1-{x}^{-1}\log(1+x), x>0,$
and $H^{**}$ is a continuously increasing probability distribution function.
One way to get a class of infinite angular measures satisfying
\eqref{eqn:condOnS} is to take $S(\mathrm{d}w) = (1-w)^{-1} F(\mathrm{d}w)$ for
probability measures $F(\cdot)$ on $[0,1)$.
\end{Example}

\section{Extending the CEV model to a multivariate extreme value
model}\label{sec:extension}
The CEV model assumes the existence of a vague limit in a {smaller}
subset of
Euclidean space than {that required} by classical MEVT.
Given a CEV model, when can it be extended to
a MEVT model?
If such an extension of the CEV model is possible, then $X$ will also
have a distribution in a DOA, so this will be
assumed. The following is a sufficient condition for
such an extension.
\begin{Proposition}{\label{thm:exten}}
Suppose that $(X,Y) $
satisfy Definition~\ref{def:conevm} and, in particular,
\eqref{eqn:cond2}--\eqref{eqn:Hispdf}.
Assume that $X
\in D(G_{\lambda})$ for some $\lambda\in\R$ so that there exist
functions $\chi(t) > 0, \phi(t) \in\R$ such that
\begin{eqnarray*}
t \P \biggl(\frac{X-\phi(t)}{\chi(t)} > x \biggr)  \to(1 + \lambda x)^{-1/\lambda},\qquad 1+\lambda x > 0 .
\end{eqnarray*}
If $\lim_{t \to\infty} \alpha(t)/\chi(t) $ exists and is finite and
both $\lim_{t \to\infty} \beta(t), \lim_{t \to\infty} \phi(t) $
exist ($\le\infty$) and are equal, then $(X,Y)$ is in the domain
of attraction of a multivariate
extreme value distribution on $\E^{(\lambda,\gamma)}$; that is,
for a Radon measure $\mu_{X,Y}(\cdot) $ on $ \E^{(\lambda,\gamma)}$,
\[
t\P\biggl[ \biggl(\frac{X-\phi(t)}{\chi(t)},
\frac{Y-b(t)}{a(t)} \biggr) \in \cdot \biggr] \cnvg
\mu_{X,Y}(\cdot)\qquad \mbox{in } \mathbb{M}_{+} \bigl(
\E^{(\lambda,\gamma)} \bigr).
\]
\end{Proposition}

\begin{pf}
For $\lambda>0$, the proof is a consequence of cases 1 and 2 of
Theorem~\ref{thm:gencon}. The other cases can be proven similarly.
\end{pf}

We now discuss the extension of the CEV model to MEVT
after first standardizing $(X,Y)$ to $(X^*,Y^*),$ which is regularly
varying on $\Enb$. We consider extending regular
variation on $\Enb$ to an \textit{asymptotically tail equivalent} regular
variation on $\E$, a notion we define next.
\begin{Definition}[(Tail equivalence in multivariate regular variation \cite{maulikresnick2005})] \label{def:tailequiv}
If $\bX$ and $\bY$ are $\R_+^d$-valued random vectors, then $\bX$
and $\bY$ are \textit{\textit{tail equivalent}} on a cone $\C
\subset\overline{\R}_{+}^d
$ if there exists a scaling function $b(t) \uparrow
\infty$ such that
\[
t\P[{\bX}/{b(t)} \in \cdot ]
\cnvg\nu(\cdot) \quad\mbox{and}\quad
t\P[{\bY}/{b(t)} \in \cdot ]
\cnvg c\nu(\cdot)
\]
in $M_{+}(\C)$ for some $c > 0$ and non-null Radon measure $\nu$ on
$\C$. We write $X \stackrel{te(\C)}{\sim} Y$.
\end{Definition}
\begin{Proposition} \label{pro:taileqv}
Suppose that $(X^*,Y^*) $ is standard regularly varying on
$\Enb$ with limit measure $\nunb$ and angular measure $\Snb$ on
$[0,1)$. The following are equivalent:
\begin{enumerate}[(1)]
\item[(1)]$\Snb$ is finite on $[0,1)$;
\item[(2)] there exists a random vector $(X^\#,Y^\#)$ defined on
$\E=[0,\infty]^2\setminus\{\bzero\}$ such that
\[
(X^\#,Y^\#) \stackrel{te(\Enb)}{\sim}
(X^*,Y^*)
\]
and $(X^\#,Y^\#)$ is multivariate regularly varying on $\E$
with limit measure $\nu$ such that $\nu|_{\Enb} = \nunb$.
\end{enumerate}
\end{Proposition}

\begin{pf}
$(1) \Rightarrow(2)$:
Define the polar coordinate transformation $(R,\Theta) = (X^*+Y^*,
\frac{X^*}{X^*+Y^*})$. From Section~\ref{subsec:polar} and
\eqref{eqn:defS},
for $r>0$
and $\Lambda$ a Borel subset of $[0,1)$, as $t \to\infty$,
\[
t\P\biggl[\frac{R}{t} > r, \Theta\in\Lambda \biggr] \to
r^{-1}\Snb(\Lambda)=
\nunb\biggl\{(x,y) \in\Enb\dvtx  x+y >
r, \frac{x}{x+y} \in\Lambda\biggr\}.
\]
Since $\Snb$ is finite on $[0,1)$, the distribution of $\Theta
$ is finite on $[0,1)$. Assume that $\Snb[0,1)=1$ so that it is a probability
measure and extend the measure $\Snb$ to $[0,1]$ by
putting $\Snb(\{1\}) =0$. Define $R_0$ and $\Theta_0$ to be
independent. $\Theta_0$ has distribution given by the
extended $\Snb$ on $[0,1]$ and $R_0$ has the standard Pareto distribution.
Define
$(X^\#,Y^\#) = (R_0\Theta_0, R_0(1-\Theta_0)),
$
so $(X^\#,Y^\#)$ is regularly varying on $\E$ with
standard scaling and limit measure $\nu$, where $\nu|_{\Enb} =
\nunb$.

$(2) \Rightarrow(1)$: Referring to \eqref{eqn:defS}, note that
$ \Snb([0,1) ) = \nunb\{(x,y) \in\Enb: x+y > 1\}.$
Since $(X^\#,Y^\#)$ is regularly varying
on $\E$, we have
\[
t\P({X^\#+Y^\#} > {t}) \to\nu\{(x,y) \in\Enb\dvtx  x+y > 1\} <
\infty.
\]
However,
\[
\nu\{(x,y) \in\Enb\dvtx  x+y > 1\} = \nunb\{(x,y) \in
\Enb\dvtx  x+y > 1\} = \Snb ([0,1) ).
\]
Hence, $\Snb$ is finite on $[0,1)$.
\end{pf}

\section{The CEV model and hidden regular variation} \label{sec:CEVHRV}

A vector $(X^*,Y^*)$ whose distribution is standard bivariate
regularly varying on $\E$
possesses \textit{hidden regular variation} (HRV) \cite{resnick2002a}
if there exists
a Radon measure $\nu_0 \neq0$
on $\E_0 =(0,\infty]\times(0,\infty]$ and a non-decreasing function
$a_0(t) \uparrow
\infty$ with $t/a_0(t) \to\infty$, such that, in $\bM_{+}(\E_0),$
\begin{eqnarray*}
t\P\bigl( a^{-1}_0(t){(X^*,Y^*)} \in \cdot \bigr) \cnvg
\nu_0(\cdot) .
\end{eqnarray*}
If hidden regular variation holds, then $X^*$ and $Y^*$ must be
asymptotically independent,
\[
\nu^{**}([\bzero,(x,y)]^c) = x^{-1} + y^{-1},\qquad x,y > 0.
\]
{Built into the definition of HRV is regular variation on $\E$; our
formulation of the CEV model, when it can be standardized, does not
require regular variation on $\E$, but only on $\E_\sqcap$. Therefore,
comparisons between HRV and the CEV model must be carefully posed.}

Suppose that $(X,Y) \in D(G)$ for a {bivariate}
extreme value distribution $G$. If $X$ and $Y$ are asymptotically
dependent, then the CEV model holds with {either $X$ or $Y$ as
conditioning variable.}
{The centering and scaling functions for CEV can be the
same ones as for MEV.}
We can standardize $(X,Y) \mapsto(X^*,Y^*)$, so
$(X^*,Y^*)$ is standard regularly varying on $\E$ with limit
measure $\nu^{**}$,
but asymptotic dependence implies that
{hidden regular variation} cannot hold.

It is possible for HRV to hold without a CEV model being valid.
\begin{Example}
Suppose that $(X^*,Y^*)$ are random variables such that for $\alpha>1$
and $x,y \ge1$,
\[
\P\bigl[ (X^*,Y^*) \in([\bzero,(x,y)])^c \bigr] = [{x}^{-1} +
y^{-1} + ({x^{\alpha} \wedge y^{\alpha}})^{-1} ]/3.
\]
Asymptotic independence and HRV hold for $(X^*,Y^*)$ with
$a_0(t)=t^{1/\alpha}$. The CEV model does not hold, whatever
normalization we choose;
if a limit holds, it is degenerate.
\end{Example}

The following are comments on the relations between MEVT, the CEV model and
HRV:
MEVT is equivalent via standardization to regular variation on $\E$.
The CEV model, if standardization is possible, is equivalent to
regular variation on $\E_\sqcap$. Hidden regular variation requires
standard regular variation on $\E$ and regular variation of lower
order on $\E_0$. {For a pair $(X^*,Y^*)$ which is standard
regularly varying on $\E$:}
\begin{itemize}
\item asymptotic dependence of $(X^*,Y^*)$ implies that the CEV model
holds, but HRV does not; {the requirement that $a_0(t)$ be of
lower order than $t$ fails;}
\item{the presence of HRV does not imply that the CEV model holds;}
\item we conjecture that if CEV holds with asymptotic independence,
then HRV must hold; this is evident in {several} examples, but we
have no proof to turn this conjecture into fact.
\end{itemize}
More on HRV and generalizations to higher dimensions can be found in
\cite{mitraresnick2010}.

\section{Examples}\label{sec:examples}

This section presents examples that illustrate how the CEV model
differs from the usual
multivariate extreme value model.
\begin{Example}\label{ex:diffcone}
This example emphasizes that different scaling and centering functions
are required for
different cones. We will consider a bivariate random vector which is
multivariate regularly varying on $\E$ with asymptotic
independence. We then show that it possesses hidden regular variation
(see Section~\ref{sec:CEVHRV}) and also CEV limits under different
scalings. Let $X,Y$ {be} i.i.d.~$\mathit{Pareto}(1)$ random variables. Let $B$
be a Bernoulli random variable with $\P(B=0)=\P(B=1) = 0.5,$ $U$
a $\operatorname{Uniform} (0,1)$ random variable and suppose that $X,Y,B,U$ are all
independent. Define
\[
\bZ=(Z_1,Z_2) = B(UX,X^2) + (1-B)(Y^2,UY).
\]
As $t \to\infty$, observe that the following hold:
\begin{longlist}[(iii)]
\item[(i)] in $\mathbb{M}_{+}(\E),$
%
\begin{equation}\label{exeq:parE}
t\P \biggl[ \frac{\bZ}{t^2} \in
([0,x]\times[0,y])^c \biggr] \to\frac12 \biggl[\frac{1}{\sqrt{x}}+\frac{1}{\sqrt
{y}} \biggr],\qquad x\wedge y >0;
\end{equation}
\item[(ii)] in $\mathbb{M}_{+}(\E_{0})$,
%
\begin{equation}\label{exeq:parE0}
t\P \biggl[ \frac{\bZ}{t} \in
(x,\infty]\times(y,\infty] \biggr] \to\frac12 \biggl[\frac1x + \frac1y \biggr],\qquad x
\wedge y >0;
\end{equation}
\item[(iii)] in $\mathbb{M}_{+}(\Enb)$, the limit is not a
product measure and we have
%
\begin{equation}\label{exeq:parEnb}
t\P \biggl[ \biggl( \frac{Z_1}{t},\frac{Z_2}{t^2} \biggr) \in
[0,x]\times(y,\infty] \biggr] \to\frac12 \biggl[\frac{1}{\sqrt{y}} -\frac1{2x}
\biggr]_{+},\qquad x \wedge y >0;
\end{equation}
\item[(iv)] similarly, in $\mathbb{M}_{+}(\Enl)$, the limit is still
not a
product measure and we have
%
\begin{equation}\label{exeq:parEnl}
t\P \biggl[ \biggl( \frac{Z_1}{t^2},\frac{Z_2}{t} \biggr) \in
(x,\infty]\times[0,y] \biggr] \to\frac12 \biggl[\frac{1}{\sqrt{x}} -\frac1{2y}
\biggr]_{+},\qquad x \wedge y >0.
\end{equation}
\end{longlist}
This provides an example for the validity of Theorem~\ref{thm:gencon}.
The example holds even if we ignore the random variable $U$, but then
the distribution of $\bZ$ concentrates on two parabolic lines
restricted to $[0,\infty)^2$. Also, note that the limit measure for (i)
concentrates on the lines
through $\bzero$, the limit measure in (ii) concentrates on the lines
through $\binfty$ and, in (iii) (and, similarly, (iv)), the limit
measure does not concentrate on
the boundaries. This final feature can also be observed in Example~\ref{ex:ex6}.
\end{Example}
\begin{Example}[(Example~\ref{ex:ex1} continued)] \label{ex:ex6}
Recall Example~\ref{ex:ex1}. We had a bivariate joint density for
$(X,Y)$ and in the different cones, we have convergence with different
normalizations (as $t \to\infty$):
\begin{eqnarray*}
\mbox{in } \mathbb{M}_{+}(\E)\dvtx  t\P\biggl( \biggl(\frac{X}{t}, \frac{Y}{t} \biggr) \in
([0,x]\times[0,y])^c \biggr) &\to&\frac{1}{x} + \frac1y,\qquad x>0, y>0,\\
\mbox{in } \mathbb{M}_{+}(\E_0)\dvtx  t\P\biggl(\frac{X}{\sqrt{t}} > x, \frac
{Y}{\sqrt{t}} > y \biggr) &\to&\frac{1}{x^2} + \frac1{y^2},\qquad x>0, y>0,\\[1pt]
\mbox{in } \mathbb{M}_{+}(\Enb)\dvtx  t\P\biggl(\frac{X}{\sqrt{t}} \le x, \frac
{Y}{t} > y\biggr) &\to&\frac{1}{y} - \frac{1}{y + x^2}, \qquad x \ge0, y > 0,\\[1pt]
\mbox{in } \mathbb{M}_{+}(\Enl)\dvtx  t\P\biggl(\frac{X}{t} > x, \frac{Y}{\sqrt
{t}}\le y\biggr) &\to&\frac{1}{x} - \frac{1}{x + y^2},\qquad x >0, y \ge0.
\end{eqnarray*}
\end{Example}
\begin{Example}
Suppose that $(X,Y)$ has the following distribution generated by an
Archimedean copula:
\begin{eqnarray*}
F(x,y) := \frac{(1-1/x)(1-1/y)}{(1+1/(xy))},\qquad x,y \ge1.
\end{eqnarray*}
Clearly, $X$ and $Y$ are marginally $\mathit{Pareto}(1)$ random variables and, for $x,y>0$,
\begin{eqnarray*}
t\P\bigl[t^{-1}{(X,Y)}\in[\bzero,(x,y)]^c\bigr]  = t\bigl(1-F(tx,ty)\bigr) \to x^{-1} +
y^{-1} \qquad (t \to\infty).
\end{eqnarray*}
Hence, asymptotic independence holds and, for $x,y>0,$
\begin{eqnarray*}
t\P[ X >{t^{1/3}}x, Y > {t^{1/3}}y ]  \to \frac{1}{xy}\biggl(\frac1x+\frac1y\biggr) \qquad (t \to\infty),
\end{eqnarray*}
which implies hidden regular variation. {We also have}
{the} CEV {model} holding with a limit product measure since
\begin{eqnarray*}
t\P[ X \le x, Y >t y ] & \to(1-x^{-1} )y^{-1} \qquad (t \to\infty).
\end{eqnarray*}
\end{Example}

\begin{Example}
This example gives a class of limit distributions on
$\Enb$ indexed by probability distributions on $[0,\infty]$. Suppose that
$R$ is a Pareto random variable on $[1,\infty)$ with parameter $1$
and $\xi $ is a random variable with distribution $G(\cdot)$ on
$[0,\infty]$. Assume that $\xi$ and
$R$ are independent and define
$ (X,Y) = (R\xi,R).$
Then, for $y>0, x \ge0 $ and $ty>1$,
\begin{eqnarray*}
t\P \biggl[ \frac{X}{t} \le x ,\frac{Y}{t} > y \biggr]
&=& t\P \biggl[\frac{R\xi}{t} \le x ,\frac{R}{t} > y \biggr]
= t\int_{ty}^{\infty} \P\biggl[ \xi \le \frac{tx}{r} \biggr] r^{-2}\, \mathrm{d}r
\\[1pt]
&=& \int_{y}^{\infty} \P\biggl[ \xi \le\frac{x}{s} \biggr] s^{-2} \,\mathrm{d}s =
\int_{y}^{\infty} G \biggl( \frac{x}{s} \biggr) s^{-2} \,\mathrm{d}s
\\[1pt]
&=& \frac{1}{x}\int_{0}^{x/y} G(s)\, \mathrm{d}s
\\[1pt]
&=& \mu \bigl([0,x]\times(y,\infty] \bigr).
\end{eqnarray*}

This can be expressed in polar coordinates. The
angular measure $S(\cdot)$ on $\Enb$
is
\[
S([0,\eta]) = \mu\biggl\{(u,v)\dvtx  u+v > 1, \frac{u}{u+v} \le\xi \biggr\},\qquad
0\le\eta< 1.
\]
Hence, we have
\begin{eqnarray*}
&&t\P\biggl[\frac{X+Y}{t} > 1, \frac{X}{X+Y} \le\eta \biggr]
\\
&&\quad=t\P\biggl[\frac{R\xi+ R}{t} > 1, \frac{R\xi}{R\xi + R} \le\eta \biggr]
=t\P\biggl[\frac{R(1+\xi )}{t} > 1,\xi \le\frac{\eta}{1-\eta} \biggr]
\\
&&\quad= t \int_{0\le s \le\eta/(1-\eta)}
\P\biggl[\frac{R}{t}(1+s) > 1\biggr] G(\mathrm{d}s)
= t \int_{0\le s \le
\eta/(1-\eta)} \biggl(\frac{t}{1+s} \vee1 \biggr)^{-1} G(\mathrm{d}s)
\\
&&\quad= \int_{0\le s \le\eta/(1-\eta)} (1+s) G(\mathrm{d}s)
\end{eqnarray*}
for $t > {1}/{(1-\eta)}$.
However, the left-hand side goes to $\mu\{(u,v)\dvtx  u+v
> 1, \frac{y}{u+v} \le\xi \} = S([0,\eta])$ as $t \to\infty$
and, thus,
\[
S([0,\eta]) = \int_{0\le s \le
\eta/(1-\eta)} (1+s) G(\mathrm{d}s),\qquad 0\le\eta< 1.
\]
Hence, $S$ is a finite angular measure if and only if $G$ has first moment.
\end{Example}

\section{Proofs}\label{sec:proofs}

In this section, we provide proofs of some of the results given in the
previous sections.

\subsection{\texorpdfstring{Proof of
Theorem~\protect\ref{thm:gencon}}{Proof of Theorem 2.1}}

Assume that $\lambda> 0, \gamma> 0$; other
cases can be dealt with similarly.
From \eqref{eqn:c1} and \eqref{eqn:c2}, respectively, we get
\begin{eqnarray}\label{eqn:m1}
t\P \biggl(\frac{Y-b(t)}{a(t)} > y \biggr) & \to&(1 + \gamma y)^{-1/\gamma},\qquad 1+\gamma y > 0, \\\label{eqn:m2}
t \P \biggl(\frac{X-\phi(t)}{\chi(t)} > x \biggr) & \to&(1 + \lambda x)^{-1/\lambda},\qquad 1+\lambda x > 0 .
\end{eqnarray}
Hence, for $(x,y) \in\E^{(\lambda)}\times\E^{(\gamma)}$, which are
continuity points of the limit measures $\mu_{X,Y>} $ and $\mu_{Y,X>}$,
\begin{eqnarray}\label{eqn:forgenconv}
Q_t(x,y)&:  =& t\P\biggl[ \biggl(\frac{X-\phi(t)}{\chi(t)},
\frac{Y-b(t)}{a(t)} \biggr) \in ([-\infty,x]\times[-\infty,y])^c \biggr] \nonumber\\
& =& t\P\biggl[\frac{X-\phi(t)}{\chi(t)} > x \biggr]
+ t\P\biggl[\frac{Y-b(t)}{a(t)}> y \biggr]\nonumber
\\[-8pt]\\[-8pt]
&&{}-t\P\biggl[\frac{X-\phi(t)}{\chi(t)} > x , \frac{Y-b(t)}{a(t) }> y
\biggr] \nonumber\\
& =& A_t(x) + B_t(y) + C_t(x,y)\qquad \mbox{(say)}.\nonumber
\end{eqnarray}
It suffices to show that $Q_t(x,y)$ has a limit and that the limit is
non-degenerate in $(x,y)$ (using a generalized version of
\cite{resnickbook2007}, Lemma 6.1). As $t \to\infty$, we have
the limits for $A_t(x)$ and $B_t(y)$ from \eqref{eqn:m2}
and \eqref{eqn:m1}. Clearly, $0 \le C_t(x,y) \le
\min(A_t(x),B_t(y))$ and these inequalities also hold for any limit
of $Q_t$.

From \cite{heffernanresnick2007}, Proposition 1, there exist
functions $\psi_1(\cdot),\psi_2(\cdot),\psi_3(\cdot),\psi_4(\cdot)$
such that
\begin{eqnarray}
\lim_{t \to\infty}  \frac{\alpha(tz)}{\alpha(t)} &=&\psi_1(z) =z^{\rho_1},\qquad
\lim_{t \to\infty} \frac{\beta(tz)- \beta(t)}{\alpha(t)} = \psi_2(z), \label{eqn:limalbet}\\
\lim_{t \to\infty}  \frac{c(tz)}{c(t)}&=& \psi_3(z)=z^{\rho_2},\qquad
\lim_{t \to\infty}  \frac{d(tz) -d(t)}{c(t)}= \psi_4(z) \label{eqn:limcd}
\end{eqnarray}
for $z>0$ and $\rho_1,\rho_2$ real. Temporarily assume that $\rho_1$ and
$\rho_2 $ are positive. Either $\psi_2(z) =
0$, which implies that $\lim_{t \to
\infty}\beta(t)/{\alpha(t)} = 0$ (from \cite
{binghamgoldieteugels1987}, Theorem 3.1.12(a,c)) or
$\psi_2(z) = k(z^{\rho_1}-1)/\rho_1$ for $ k \neq0$, which means that
$\lim_{t \to\infty}{\beta(t)}/{\alpha(t)} =
{k}/{\rho_1}$ (\cite{dehaanferreira2006}, Proposition B.2.2).
Hence, allowing the constant $k$ to be zero as well, we can write
both cases as $\lim_{t \to\infty}{\beta(t)}/{\alpha(t)}
= {k_1}/{\rho_1}$ for some $k_1 \in\R$. Similarly, we
have $\lim_{t \to\infty}{d(t)}/{c(t)} =
{k_2}/{\rho_2}$ for some $k_2 \in\R$.

Additionally,
marginal DOA conditions for $X,Y$ yield
($z>0, w>0$)
%
\begin{equation}\label{eqn:aisrvgamma}
\lim_{t \to\infty}\frac{b(tz)- b(t)}{a(t)} =\frac{z^{\gamma}-1}{\gamma} ,\qquad
\lim_{t \to\infty}\frac{\phi(tw)-\phi(t)}{\chi(t)}=\frac{w^{\lambda}-1}{\lambda},
\end{equation}
which imply
%
\begin{equation}\label{eqn:chiisrvlambda}
\lim_{t \to\infty}\frac{a(tz)}{a(t)} =z^{\gamma},\qquad
\lim_{t \to\infty} \frac{\chi(tw)}{\chi(t)}= w^{\lambda}.
\end{equation}
Observe that
\begin{eqnarray}\label{eqn:cwithx}
C_t(x,y) & =& t\P\biggl[\frac{X-\phi(t)}{\chi(t)} > x ,
\frac{Y-b(t)}{a(t) }> y \biggr] \nonumber\\[-8pt]\\[-8pt]
& =& t\P\biggl[\frac{X-\beta(t)}{\alpha(t)} > \biggl(x + \frac{\phi(t)}{\chi(t)} \biggr)
\frac{\chi(t)}{\alpha(t)} - \frac{\beta(t)}{\alpha(t)} , \frac{Y-b(t)}{a(t) }> y\biggr]\nonumber
\end{eqnarray}
and also
\begin{eqnarray}
C_t(x,y)  =
t\P\biggl[\frac{X-\phi(t)}{\chi(t)} > x , \frac{Y-d(t)}{c(t) }>
\biggl(y + \frac{b(t)}{a(t)} \biggr)\frac{a(t)}{c(t)} - \frac{
d(t)}{c(t)} \biggr]. \label{eqn:cwithy}
\end{eqnarray}
From \cite{dehaanferreira2006}, Proposition
B.2.2, we have that
\begin{eqnarray}
 {b(t)}/ {a(t)} \to
{1}/{\gamma} \quad\mbox{and}\quad {\phi(t)}/{\chi(t)} \to
{1}/{\lambda}.\label{eqn:bapc}
\end{eqnarray}

We analyze $C_t(x,y) $ for the different cases. First, we will
show that at least one of the limits $\lim_{t \to
\infty}\frac{\chi(t)}{\alpha(t)}$ and $\lim_{t \to\infty}
\frac{a(t)}{c(t)} $ must exist. Suppose both do not exist.
We have, for $(x,y) \in\E^{(\lambda)}\times\E^{(\gamma)}$, which are
continuity points of the limit measures $\mu_{X,Y>} $ and $\mu_{Y,X>}$,
\begin{eqnarray}
t\P\biggl[\frac{X- \beta(t)}{\alpha(t)}> x, \frac{Y - b(t)}{a(t)} >y \biggr] & \to&
\mu_{X,Y>}\bigl((x,\infty] \times(y,\infty]\bigr), \label{eqn:c1xy}\\
t\P\biggl[\frac{X- \phi(t)}{\chi(t)}> x, \frac{Y - d(t)}{c(t)} >y \biggr] & \to&\mu_{Y,X>}\bigl((x,\infty] \times(y,\infty]\bigr).
\label{eqn:c2xy}
\end{eqnarray}
Now, \eqref{eqn:c1xy} implies that
\begin{eqnarray*}
&&t\P\biggl[\frac{X- \phi(t)}{\chi(t)} \frac{\chi(t)}{\alpha(t)} +
\frac{\phi(t)- \beta(t)}{\alpha(t)}> x, \frac{Y -
d(t)}{c(t)}\frac{c(t)}{a(t)} + \frac{d(t)-b(t)}{a(t)}
> y \biggr] \\
&&\quad\to
\mu_{X,Y>} \bigl((x,\infty] \times(y,\infty] \bigr),
\end{eqnarray*}
{which is equivalent to}
\begin{eqnarray*}
&&t\P\biggl[\frac{X-\phi(t)}{\chi(t)} > \frac{\alpha(t)}{\chi(t)}
\biggl(x-\frac{\phi(t)- \beta(t)}{\alpha(t)} \biggr), \frac{Y - d(t)}{c(t)}
> \frac{a(t)}{c(t)} \biggl(y - \frac{d(t)-b(t)}{a(t)} \biggr) \biggr] \\
&&\quad\to \mu_{X,Y>} \bigl((x,\infty] \times(y,\infty] \bigr).
\end{eqnarray*}
From \eqref{eqn:c2xy}, we also have that the left-hand side of the
previous line has a limit
\begin{eqnarray*}
&&t\P\biggl[\frac{X- \phi(t)}{\chi(t)} > \frac{\alpha(t)}{\chi(t)} \biggl(x -\frac{\phi(t)-
\beta(t)}{\alpha(t)} \biggr), \frac{Y - d(t)}{c(t)}> \frac{a(t)}{c(t)} \biggl(y - \frac{d(t)-b(t)}{a(t)} \biggr) \biggr]\\
&&\quad \to\mu_{Y,X>} \bigl((f(x),\infty] \times(g(y),\infty] \bigr)
\end{eqnarray*}
for some $(f(x),g(y))$, assumed to be a continuity point of the limit
$\mu_{Y,X>}$, if and only if, as $t\to\infty$, the following two
limits hold:
\begin{eqnarray}
\frac{\alpha(t)}{\chi(t)} \biggl(x -\frac{\phi(t)- \beta(t)}{\alpha(t)} \biggr) &\to& f(x), \label{eqn:limf}\\
\frac{a(t)}{c(t)} \biggl(y -\frac{d(t)- b(t)}{a(t)} \biggr) &\to& g(y) . \label{eqn:limg}
\end{eqnarray}
For $\mu_{Y,X>}$ to be non-degenerate, $f$ and $g$ should be non-constant
and we should also have $\mu_{X,Y>} ((x,\infty] \times(y,\infty] ) =
\mu_{Y,X>} ((f(x),\infty] \times(g(y),\infty] )$. Considering
\eqref{eqn:limf} and \eqref{eqn:limg}, we can see that the limit as
$t \to\infty$ exists if and only if $ \lim_{t \to\infty}
{a(t)}/{c(t)} $ and $ \lim_{t \to\infty}
{\chi(t)}/{\alpha(t)} $ exists.

We conclude that $ \lim_{t \to\infty}{\chi(t)}/{\alpha(t)} \in
[0,\infty]$ and consider the following cases:
\begin{itemize}
\item\textit{Case} 1: $\lim_{t \to\infty}{\chi(t)}/{\alpha(t)} =
\infty$.
Consider \eqref{eqn:cwithx} and note
that
\[
\biggl(x + \frac{\phi(t)}{\chi(t)} \biggr) \frac{\chi(t)}{\alpha(t)} -
\frac{\beta(t)}{\alpha(t)} \to \biggl(x+ \frac{1}{\lambda} \biggr)\times
\infty- \frac{k_1}{\rho_1} = \infty,
\]
which entails that
$\lim_{t \to\infty} C_t(x,y) = \mu_{X,Y>} (\{\infty\}
\times(y,\infty]) = 0.$ Hence,
\[
\lim_{t \to\infty} Q_t(x,y)
= (1+\lambda x)^{-1/\lambda} + (1+\gamma y)^{-1/\gamma}.
\]
\item\textit{Case 2}: $\lim_{t \to\infty}{\chi(t)}/{\alpha(t)} = M \in
(0,\infty)$.
From \eqref{eqn:cwithx}, we have
\[
\biggl(x + \frac{\phi(t)}{\chi(t)} \biggr) \frac{\chi(t)}{\alpha(t)} - \frac{\beta
(t)}{\alpha(t)} \to \biggl(x+ \frac{1}{\lambda} \biggr)\times M - \frac{k_1}{\rho
_1} =f(x) \qquad\mbox{(say)}.
\]
Therefore,
\[
\lim_{t \to\infty} C_t(x,y) = \mu_{X,Y>} \bigl((f(x),\infty]
\times(y,\infty]\bigr) \le(1+\lambda y)^{-1/\lambda}
\]
with strict
inequality holding for some $x$ because of the non-degeneracy
condition \eqref{eqn:nondegx} for $\mu_{X,Y>}$. Hence,
\[
\lim_{t \to
\infty} Q_t(x,y) = (1+\lambda x)^{-1/\lambda} + (1+\gamma
y)^{-1/\gamma} - \mu_{X,Y>} \bigl((f(x),\infty] \times(y,\infty]\bigr).
\]
\item\textit{Case} 3: $\lim_{t \to\infty}{\chi(t)}/{\alpha(t)} = 0$.
In this case, \eqref{eqn:cwithx} leads to a
degenerate limit in $x$ for $C_t(x,y)$ and putting $M_1 =
{k}/{\rho_1}$, we get
\begin{eqnarray*}
\lim_{t \to\infty} C_t(x,y) = \mu_{X,Y>} \bigl((M_1,\infty] \times
(y,\infty]\bigr) =: f_1(y) \le(1+\gamma y)^{-1/\gamma}.
\end{eqnarray*}
So, consider \eqref{eqn:cwithy}.
\begin{enumerate}[(1)]
\item[(1)] If $\lim_{t \to
\infty} {a(t)}/{c(t)} $ exists in $(0,\infty]$, then we can
use a similar technique as in case~1 or~2 to obtain a
non-degenerate limit for $Q_t(x,y)$.
\item[(2)] If $\lim_{t\to\infty} {a(t)}/{c(t)} =0$, then for some $M_2
\in\R$,
\begin{eqnarray*}
\lim_{t \to\infty} C_t(x,y)  = \mu_{Y,X>}
\bigl((x,\infty] \times(M_2,\infty]\bigr) =: f_2(x) \le(1+\lambda
x)^{-1/\lambda}.
\end{eqnarray*}
 Therefore, we have, for any $(x,y)\in
\E^{(\lambda)}\times\E^{(\gamma)}$ which are continuity points of
the limit measures $\mu_{X,Y>} $ and $\mu_{Y,X>}$,
\begin{eqnarray*}
f_1(y) = \mu_{X,Y>}
\bigl((M_1,\infty]
\times(y,\infty]\bigr)  = \mu_{Y,X>} \bigl((x,\infty] \times(M_2,\infty]\bigr) = f_2(x).
\end{eqnarray*}
It is now easy to check that for any $(x,y)\in\E^{(\lambda)}\times
\E^{(\gamma)}$ which are continuity points of the limit measures
$\mu_{X,Y>} $ and $\mu_{Y,X>}$, we have $f_1(y) = f_2(x) = 0$. Hence,
$C_t(x,y) \to
0$ and thus $Q_t(x,y)$ has a non-degenerate limit.
\end{enumerate}
\end{itemize}
This proves the result. For general $\rho_1,\rho_2 \in\R$, we can
follow the same steps to get to the result by considering cases when
$\rho_i$ is greater than, less than or equal to zero, for each $i=1,2$.

\subsection{\texorpdfstring{Proof of the summary following
Theorem~\protect\ref{thm:productsBad}}{Proof of the summary following Theorem 3.2}}

\hspace*{7.5pt} (1) This part has been dealt with in \cite{heffernanresnick2007}, Section
2.4.

(2) First, simplify the problem. For
$(x,y)$, a continuity point of $\mu(\cdot)$,
\[
t \P \biggl[\frac{\lambda(X^*)-\beta(t)}{\alpha(t)} \le x,
\frac{b(Y^*)-b(t)}{a(t)} > y \biggr] \to
\tilde{\mu} ([-\infty,x] \times
(y,\infty] )\qquad (t\to\infty)
\]
is equivalent, as $t\to\infty$, to
%
\begin{eqnarray}\label{eqn:stdformthm}
t \P \biggl(\frac{\lambda(X^*)-\beta(t)}{\alpha(t)} \le x,
\frac{Y^*}{t} > y \biggr) &\to&\tilde{\mu} \bigl([-\infty,x] \times
(h(y),\infty] \bigr)\nonumber
\\[-8pt]\\[-8pt]
&=:& \mu^{*} \bigl([-\infty,x] \times
(y,\infty] \bigr),\nonumber
\end{eqnarray}
where
%
\begin{equation}
h(y) =
\cases{
(1+\gamma y)^{1/\gamma}, &\quad$\gamma \neq 0$,\cr
\mathrm{e}^y, &\quad$\gamma = 0$.
}
\end{equation}
Hence, $\eqref{eqn:trans}$ is equivalent to
\begin{eqnarray*}
t \P \biggl[ \biggl(\frac{\lambda(X^*)-\beta(t)}{\alpha(t)}, \frac{Y^*}{t} \biggr) \in
\cdot \biggr] & \cnvg\mu^{*}(\cdot)
\end{eqnarray*}
and $\mu^{*}$ is a non-null Radon measure on
$[-\infty,\infty]\times\overline\E^{(\gamma)}$ satisfying the
conditional non-degeneracy conditions. Hence, our proof will show
the existence of $\lambda(\cdot)$ satisfying \eqref{eqn:stdformthm}.
Now, note that \eqref{eqn:condpsi1psi2} implies that
$\alpha(\cdot) \in RV_{\rho}$ for some $\rho\in\R$ and $\psi_1(x)
= x^{\rho}$ (\cite{resnickbook2008}, page 14). The function
$\psi_2(\cdot)$ may be identically equal to 0 or

%
\begin{equation} \label{eqn:psi2}
\psi_2(x) =
\cases{
k(x^{\rho}-1)/\rho&\quad \mbox{if }$\rho\neq0, x > 0$, \cr
k \log x &\quad \mbox{if }$\rho= 0, x > 0$
}
\end{equation}
for $k \neq0$ (\cite{dehaanferreira2006}, page 373). We have
assumed that $(\psi_1,\psi_2) \neq(1,0)$. We will consider three
cases: $ \rho>0, \rho=0, \rho<0$.

\textit{Case} 1: $\rho> 0$.
First, suppose that $\psi_2 \equiv0$.
Since $\alpha(\cdot) \in RV_{\rho}$, there exists
$\tilde{\alpha}(\cdot) \in RV_{\rho}$ which is ultimately
differentiable and
strictly increasing and $\alpha\sim
\tilde{\alpha}$ (\cite{dehaanferreira2006}, page 366). Thus,
$\tilde{\alpha}^{\leftarrow}$ exists. Additionally, from
\cite{binghamgoldieteugels1987}, Theorem 3.1.12(a), we have that
$\beta(t)/\alpha(t) \to0$. Hence, for $x > 0$, as
$t\to\infty$,
we have
\begin{eqnarray*}
 \frac{\tilde{\alpha}(tx) + \beta(t)}{\alpha(t)} = \frac{\tilde{\alpha
}(tx)}{\tilde{\alpha}(t)}\cdot\frac{\tilde{\alpha}(t)}{\alpha(t)} +
\frac{\beta(t)}{\alpha(t)} \to x^{\rho}
\end{eqnarray*}
and inverting, we get, for $ z > 0,$
\begin{eqnarray*}
 {\tilde{\alpha}^{\leftarrow} \bigl(\alpha(t)z +
\beta(t) \bigr)}/{t} \to z^{1/\rho}\qquad (t\to\infty).
\end{eqnarray*}
Thus, we have\vspace*{-2pt}
\begin{eqnarray*}
t \P \biggl[\frac{\tilde{\alpha}(X^*)-\beta(t)}{\alpha(t)} \le x,
\frac{Y^*}{t} > y \biggr] &=&t \P \biggl[\frac{X^*}{t} \le
\frac{\tilde{\alpha}^{\leftarrow}(\alpha(t)x+\beta(t))}{t},
\frac{Y^*}{t} > y \biggr] \\[-2pt]
& \to&\mu^{**} \bigl([0,x^{1/\rho}] \times(y,\infty] \bigr).
\end{eqnarray*}
Set $ \lambda(\cdot) = \tilde{\alpha}(\cdot) $ and this defines
$\tilde\mu$.

Next, suppose that $\psi_2 \neq0$. Therefore,\vspace*{-2pt}
\[
\psi_2(x) = \lim_{t \to\infty} {\bigl(\beta(tx) -
\beta(t) \bigr)}/{\alpha(t)}= k(x^{\rho}-1)/\rho,
\]
that is, $\beta(\cdot) \in
RV_{\rho}$ and $k>0$. There exists $\tilde\beta$ which is
ultimately differentiable, strictly
increasing and such that $ \tilde{\beta} \sim\beta$
(\cite{dehaanferreira2006}, page 366). Thus,
$\tilde{\beta}^{\leftarrow}$ exists. We then have, for $x > 0$, as $t\to
\infty$,\vspace*{-2pt}
\begin{eqnarray*}
\frac{\tilde{\beta}(tx) - \beta(t)}{\alpha(t)}
& =& \frac{\tilde{\beta}(tx) - \beta(tx)}{\alpha(t)} + \frac{\beta(tx) - \beta(t)}{\alpha(t)}\\
& =& \frac{\tilde{\beta}(tx) - \beta(tx)}{\beta(tx)} \frac{\beta
(tx)}{\alpha(tx)}\frac{\alpha(tx)}{\alpha(t)} + \frac{\beta(tx) - \beta
(t)}{\alpha(t)}\\
& \to& (1-1)\cdot
x^{\rho}/{\rho} + k
{(x^{\rho}-1)}/{\rho}
= k{(x^{\rho}-1)}/{\rho}.
\end{eqnarray*}
Inverting, we get, as $t\to\infty$,\vspace*{-2pt}
\[
{\tilde{\beta}^{\leftarrow} \bigl(\alpha(t)x +
\beta(t) \bigr)}/{t} \to(1+{\rho x}/{k})^{1/\rho}.
\]
Thus, we have\vspace*{-2pt}
\begin{eqnarray*}
t \P \biggl[\frac{\tilde{\beta}(X^*)-\beta(t)}{\alpha(t)} \le x,
\frac{Y^*}{t} > y \biggr] &=& t \P \biggl[\frac{X^*}{t} \le
\frac{\tilde{\beta}^{\leftarrow}(\alpha(t)x+\beta(t))}{t},
\frac{Y^*}{t} > y \biggr] \\[-2pt]
&\to&\mu^{**} \biggl(\biggl[0,\biggl(1+\frac{\rho x}{k}\biggr)^{1/\rho}\biggr] \times(y,\infty] \biggr).
\end{eqnarray*}
Here, we can set $ \lambda(\cdot) = \tilde{\beta}(\cdot) $ and this
defines $\tilde\mu$.

\textit{Case} 2: $\rho = 0$.
We have $ \psi_1(x) = 1, \psi_2(x) = k\log x$ for $ x
> 0 $ and some $k \in\R$. By assumption, $(\psi_1,\psi_2)\neq(1,0)$
and hence $k \ne0$.
First, assume that $k>0$, which means that $\beta\in
\Pi_{+}(\alpha)$. There exists
$\tilde{\beta}(\cdot)$ which is continuous, strictly increasing and
$\beta-\tilde{\beta} = \mathrm{o}(\alpha)$ (\cite{dehaanresnick1979b}, page
1031). If $\beta(\infty) =
\tilde{\beta}(\infty) = \infty$, then, for $x > 0$,
\[
\frac{\tilde{\beta}(tx) - \beta(t)}{\alpha(t)} = \frac{\tilde{\beta
}(tx) - \beta(tx)}{\alpha(tx)} \frac{\alpha(tx)}{\alpha(t)} + \frac
{\beta(tx) - \beta(t)}{\alpha(t)}
\to0 + k \log x
\]
and, inverting, we get for $z \in\R$, as $t\to\infty$,
$ {\tilde{\beta}^{\leftarrow} (\alpha(t)z +
\beta(t) )}/{t} \to\exp\{z/k\}.$
Thus, we have
\begin{eqnarray*}
t \P \biggl(\frac{\tilde{\beta}(X^*)-\beta(t)}{\alpha(t)} \le x, \frac
{Y^*}{t} > y \biggr) & =& t \P \biggl(\frac{X^*}{t} \le\frac{\tilde{\beta
}^{\leftarrow}(\alpha(t)x+\beta(t))}{t}, \frac{Y^*}{t} > y \biggr) \\
& \to&\mu\bigl([0,\mathrm{e}^{k/x}] \times(y,\infty] \bigr).
\end{eqnarray*}
If $\beta(\infty)= \tilde{\beta}(\infty)=B < \infty,$ define
\[
\beta^{*}(t) = \frac{1}{B- \tilde{\beta}(t)},\qquad
\alpha^{*}(t) = \frac{\alpha(t)}{(B- \tilde{\beta}(t))^2}
\]
and
we have that $\beta^{*} \in\Pi_{+}(\alpha^*) $, $\beta^*(t) \to
\infty$ and ${(B-\tilde{\beta}(t) )}/{\alpha(t)} \to\infty$ (\cite
{gelukdehaan1987}, page 25).
Hence, we have reduced the problem to the previous case, which implies that
\[
t \P \biggl(\frac{\beta^*(X^*)-\beta^*(t)}{\alpha^*(t)} \le x,
\frac{Y^*}{t} > y \biggr) \to\mu\bigl([0,\mathrm{e}^{k/x}] \times
(y,\infty] \bigr)
\]
or, equivalently,
\[
t \P \biggl(\frac{\tilde{\beta}(X^*)-\tilde{\beta}(t)}{\alpha(t)} \le
\frac{x}{1+\alpha(t)x/(B-\tilde{\beta}(t))}, \frac{Y^*}{t} > y
\biggr) \to\mu\bigl([0,\mathrm{e}^{k/x}] \times(y,\infty] \bigr),
\]
and since ${B-\tilde{\beta}(t)}/{\alpha(t)} \to
\infty$ implies ${\alpha(t)}/{B-\tilde{\beta}(t)} \to0$,
we can write
\[
t \P \biggl(\frac{\tilde{\beta}(X^*)-\tilde{\beta}(t)}{\alpha(t)} \le x,
\frac{Y^*}{t} > y \biggr) \to\mu\bigl([0,\mathrm{e}^{k/x}] \times(y,\infty] \bigr) ,
\]
which implies, since $ \beta-\tilde{\beta} = \mathrm{o}(\alpha)$,
that
\[
t \P \biggl(\frac{\tilde{\beta}(X^*)-\beta(t)}{\alpha(t)} \le x, \frac
{Y^*}{t} > y \biggr) \to\mu\bigl([0,\mathrm{e}^{k/x}] \times(y,\infty] \bigr).
\]
We have thus produced the required transformation $ \lambda(\cdot) =
\tilde{\beta}(\cdot) $.

The case for which $k < 0$, that is, $\beta\in\Pi_{-}(\alpha)$, can be
proven similarly.

\textit{Case} 3:  $\rho< 0$. This case is similar to the case for
$\rho>0$ and is therefore omitted.

\section*{Acknowledgments}
This research was partially supported by ARO Contract W911NF-07-1-0078
at Cornell University. We wish to thank the conscientious referees for
their helpful suggestions.

\printhistory

\end{document}